\documentclass[12pt]{amsart}
\usepackage{amsmath,mathtools,amsthm,amsfonts,bigints,amssymb, mathrsfs}

\usepackage{xcolor}
\RequirePackage[bookmarks, bookmarksopen=true, plainpages=false, pdfpagelabels, pdfpagelayout=SinglePage, breaklinks = true]{hyperref}
\usepackage[noabbrev,capitalize,nameinlink]{cleveref}
\hypersetup{
	colorlinks,
	linkcolor={blue!50!black},
	citecolor={blue!50!black},
	urlcolor={blue!50!black}
}
\makeatletter
\@namedef{subjclassname@2020}{\textup{2020} Mathematics Subject Classification}
\makeatother

\newtheorem{theorem}{Theorem}[section]
\newtheorem{lemma}[theorem]{Lemma}

\newtheorem{remark}[theorem]{Remark}

\newtheorem{question}[theorem]{Question}

\newcommand{\R}{\mathbb R}%
\newcommand{\N}{\mathbb N}%
\newcommand{\LL}{\mathcal L}
\newcommand{\G}{\mathbb G}%
\newcommand{\dd}{{\bf{d}}}
\newcommand{\ddd}{\mathfrak{d}}
\numberwithin{equation}{section}

\begin{document}
\title[Heat equation on Stratified Lie groups]{Initial singularities of positive solutions of the Heat equation on Stratified Lie groups}
\author[Utsav Dewan]{Utsav Dewan}
\address{Stat-Math Unit, Indian Statistical Institute, 203 B. T. Rd., Kolkata 700108, India}
\email{utsav\_r@isical.ac.in ;\: utsav97dewan@gmail.com }

\begin{abstract} 
Let $(\mathbb{G},\circ)$ be a stratified Lie group. We estimate the Hausdorff dimension (with respect to the Carnot-Carath\'eodory metric) of the singular sets in $\mathbb{G}$, where a positive solution of the Heat equation corresponding to a sub-Laplacian, blows up faster than a prescribed rate along normal limits, in terms of the homogeneous dimension of $\mathbb{G}$ and the rate of the blowup parameter. This generalizes a classical result of Watson for the Euclidean Heat. We also obtain the corresponding sharpness result, which is new even for $\mathbb{R}^n$.  
\end{abstract}
\maketitle
\setcounter{tocdepth}{1}
\tableofcontents
\section{Introduction}
Let $(\G,\circ)$ be a stratified Lie group (also known as a Carnot group) with a bi-invariant Haar measure $m$ and $\{X_1,\cdots,X_{N_1}\}$ be a collection of left invariant vector fields that form a basis of the first stratification layer of the Lie algebra $\mathfrak{g}$. We consider the Heat equation 
\begin{equation} \label{heat_eqn}
\frac{\partial}{\partial t} u(x,t)=\LL u(x,t)\:,\: (x,t) \in \G \times (0,\infty)\:,
\end{equation}
where $\LL$ is a sub-Laplacian on $\G$ given by,
\begin{equation*}
\LL = \displaystyle\sum_{j=1}^{N_1}  X^2_j\:.
\end{equation*}

Ever since the celebrated paper of Fatou \cite{Fatou} in 1906,  analysts and probability theorists have been deeply fascinated with the boundary behaviour of solutions of Partial Differential Equations (for instance see \cite{Gehring, K, RU, Doob1, Doob2} and the articles citing them). In this direction, Bonfiglioli et. al. obtained an analogue of the classical Fatou-Kato theorem on $\G$ by showing that \cite[Theorem 1.1]{BU}: positive solutions of the Heat equation (\ref{heat_eqn}) admit finite normal limit as $t \to 0+$, at almost every point on the boundary $\G$. 

This gives rise to a natural question: how does a positive solution of the Heat equation behave along normal limits, on the complement of this full measure subset of $\G$. More precisely,
\begin{question} \label{q1}
How fast can a positive solution of the Heat equation grow?
\end{question}
\begin{question} \label{q2}
How large can the `singular set' (in $\G$) be where a positive solution of the Heat equation blows up faster than a prescribed rate?
\end{question}
To address these questions, we first note that a positive solution of the Heat equation $u$ on $\G \times (0,\infty)$, is represented by a unique Radon measure $\mu$ on $\G$, in the following way \cite[Theorem 1.1]{BU} :
\begin{equation} \label{rep_formula}
u(x,t)=\Gamma[\mu](x,t)= \int_{\G} \Gamma(\xi^{-1} \circ x,t)\: d\mu(\xi) \:, 
\end{equation}
such that the integral is well-defined for all $(x,t) \in  \G \times (0,\infty)$, where $\Gamma$ is the Heat kernel corresponding to (\ref{heat_eqn}). Now, for any $x \in \G$ and any $t \in (0,1)$, repeated applications of the pointwise estimates of $\Gamma$ given by (\ref{Gaussian_estimates}) in (\ref{rep_formula}) yield a positive constant $c$ such that
\begin{equation} \label{kernel_general_estimate}
u(x,t) \le \left(c^{\left(\frac{Q}{2}+1\right)}\: u(x,c)\right)\: t^{-\frac{Q}{2}} \:,
\end{equation}
where $Q$ is the homogeneous dimension of $\G$ (for unexplained terminologies, we refer to section $2$).

Then as $0< u(x,c) < +\infty$,
\begin{itemize}
\item The inequality (\ref{kernel_general_estimate}) answers Question \ref{q1}.
\item  Moreover, in order to answer Question \ref{q2}, we are motivated by (\ref{kernel_general_estimate}) to consider the following sets:
\end{itemize}
\begin{eqnarray*}
&&S_\alpha(u) := \{x \in \G \mid \displaystyle\limsup_{t \to 0+} t^{\alpha/2} u(x,t)=+\infty\} \:, \text{ for } 0 \le \alpha < Q\:,\\
&&T_\alpha(u) := \{x \in \G \mid \displaystyle\limsup_{t \to 0+} t^{\alpha/2} u(x,t)>0\} \:, \text{ for } 0 \le \alpha \le Q\:.
\end{eqnarray*}

In this article, we estimate the $\beta$-dimensional outer Hausdorff measure $\mathcal{H}^\beta$ or the Hausdorff dimension $dim_{\mathcal{H}}$ of the above `singular sets', with respect to the Carnot-Carath\'eodory metric corresponding to $\{X_1,\cdots,X_{N_1}\}$: 
\begin{theorem} \label{Hausdorff_bound}
Let $\alpha \in [0,Q]$ and $u$ be a positive solution of the Heat equation on $\G \times (0,\infty)$. Then 
\begin{equation} \label{size_estimates}
\mathcal{H}^{Q-\alpha}\left(S_\alpha(u)\right)=0\:,\text{ and } dim_{\mathcal{H}}\left(T_\alpha(u)\right) \le Q-\alpha\:.
\end{equation}
\end{theorem}

\begin{remark}
\begin{itemize}
\item[(i)] Theorem \ref{Hausdorff_bound} generalizes a classical result of Watson for positive solutions of the Heat equation on $\R^n$ \cite[Theorem 6]{W}.
\item[(ii)] The key estimate in the proof of Theorem \ref{Hausdorff_bound} is given by Lemma \ref{lemma2}, which relates small time asymptotics of positive solutions of the Heat equation, with the upper density of its boundary measure. Such an estimate for $\R^n$ was proved in \cite[Theorem 2]{W} and it crucially used the explicit expression of the Euclidean Heat kernel and also involved the Laplace transform.
\item[(iii)] In the absence of the above two ingredients mentioned in $(ii)$, in the generality of Carnot groups, we obtain the estimate in Lemma \ref{lemma2}, only using pointwise estimates of $\Gamma$, given by (\ref{Gaussian_estimates}) and employing soft arguments involving integration-by-parts.
\end{itemize}
\end{remark}
Upon a brief glance at Theorem \ref{Hausdorff_bound}, an inquisitive mind is naturally led to ask whether the size estimates obtained in (\ref{size_estimates}) are indeed sharp. More precisely,
\begin{question} \label{q3}
Given $\alpha \in [0,Q)$ and any $E \subset \G$ with $\mathcal{H}^{Q-\alpha}(E)=0$, does there exist a positive solution of the Heat equation $u$ on $\G \times (0,\infty)$, such that $E \subset S_\alpha(u)$\:?
\end{question}
\begin{question} \label{q4}
Given $\alpha \in [0,Q)$, does there exist a positive solution of the Heat equation $u$ on $\G \times (0,\infty)$, such that $dim_{\mathcal{H}}\left(T_\alpha(u)\right) = Q-\alpha$\:?
\end{question}

These questions for the Heat equation, seem to be unaddressed even in the classical setting of $\R^n$. Our next result answers Questions \ref{q3} and \ref{q4} in the affirmative, in the general setting of a Carnot group:
\begin{theorem} \label{sharp_thm}
Let $\alpha \in [0,Q)$. Then
\begin{itemize}
\item[(i)] For any $E \subset \G$ with $\mathcal{H}^{Q-\alpha}(E)=0$, there exists a positive solution of the Heat equation $u$ on $\G \times (0,\infty)$, such that $E \subset S_\alpha(u)$.
\item[(ii)] There exists a positive solution of the Heat equation $u$ on $\G \times (0,\infty)$, such that $dim_{\mathcal{H}}\left(T_\alpha(u)\right) = Q-\alpha$\:.
\end{itemize} 
\end{theorem}
For the sharpness result Theorem \ref{sharp_thm},
\begin{itemize}
\item We first construct the building blocks of the divergence by means of a `uniform divergence' estimate for small time in Lemma \ref{heat_indicator_estimate} and then suitably glue them together to obtain part $(i)$.
\item For part $(ii)$, we make use of the construction made in part $(i)$ and combine it with a non-trivial, abstract result on existence of sets of desired Hausdorff dimensions (see Lemma \ref{existence}).
\end{itemize}

\begin{remark}
Theorem \ref{sharp_thm} is new even for $\R^n$ and yields the sharpness of \cite[Theorem 6]{W}\:. 
\end{remark}

This article is organized as follows. In Section $2$, we recall the required preliminaries. Theorems \ref{Hausdorff_bound} and \ref{sharp_thm} are proved in Sections $3$ and $4$ respectively. 

\section{Preliminaries}
In this section, we recall the required preliminaries and fix our notations.
\subsection{Some notations:}
Throughout, the symbol `c' will denote positive constants whose values may change on each occurrence. The enumerated constants $c_1, c_2, \dots$ will however be fixed throughout. $\N$ will denote the set of positive integers. For non-negative functions $f_1,\:f_2$ we write, $f_1 \lesssim f_2$ if there exists a constant $c \ge 1$, so that
\begin{equation*}
f_1 \le c f_2 \:.
\end{equation*} 
The indicator function of a set $A$ will be denoted by $\chi_A$. 
\subsection{Stratified Lie groups and the Heat kernel:} In this subsection, we recall the basic notions regarding Stratified Lie Groups and can be found in \cite{BLU, Folland}.

A stratified Lie group $(\G,\circ)$ (also known as a Carnot group) is a connected, simply connected nilpotent Lie group whose Lie algebra $\mathfrak{g}$ admits a vector space decomposition
\begin{equation*}
\mathfrak{g}=V_1\oplus V_2\oplus\cdots\oplus V_l,
\end{equation*}
such that
\begin{equation*}
[V_1,V_j]=V_{j+1},\:\:1\leq j<l,\:\:\:\:\:\:\:\:[V_1,V_l]=0.
\end{equation*}
Here,
\begin{equation*}
[V_1,V_j]=\text{span}~\{[X,Y]\mid X\in V_1, Y\in V_j \}.
\end{equation*}
Therefore, the first layer of stratification, $V_1$ generates $\mathfrak{g}$ as a Lie algebra. $\G$ is then said to be of step $l$. $\R^n$ is the trivial example of such groups with $l=1$. The simplest nontrivial example of such a group is given by the Heisenberg group $\mathbb{H}^n$ and in this case $l=2$.

The Lie algebra $\mathfrak{g}$ admits a canonical family of dilations denoted by $\{\delta_r\}_{r>0}$. They are Lie algebra automorphisms and are defined by \cite[P.5]{Folland}:
\begin{equation*}
\delta_r\left(\sum_{j=1}^{l}X_j\right)=\sum_{j=1}^{l}r^jX_j,\:\:X_j\in V_j.
\end{equation*}
As $\mathfrak{g}$ is nilpotent, the Lie exponential map defines a diffeomorphism from the Lie algebra onto the group, $\exp:\mathfrak{g}\to \G$. The group identity will be denoted by $0$. Now the dilations $\delta_r$ lift via the exponential map to give a one-parameter group of automorphisms of $\G$ which will also be denoted by $\delta_r$. The bi-invariant Haar measure $m$ on $\G$, introduced in the Introduction, is the push-forward of the Lebesgue measure on $\mathfrak{g}$ via $\exp$. In fact $m$ is the Lebesgue measure of the underlying Euclidean space. 

A notion intimately connected with the concept of stratification, is the Homogeneous dimension, which is denoted by
\begin{equation*}
Q=\sum_{j=1}^lj(\text{dim}\:V_j)\:.
\end{equation*}
The term `homogeneous dimension' is motivated by the following property
\begin{equation} \label{homogenity}
m\left(\delta_r(E)\right)=r^Qm(E),
\end{equation}
which holds for all measurable sets $E\subset \G$ and $r>0$. A continuous function $\dd:\G\to[0,\infty)$ is called a homogeneous norm on $\G$ if it  satisfies the following properties:
\begin{enumerate}
	\item[(i)]$\dd$ is smooth on $\G\setminus\{0\}$,
	\item[(ii)] $\dd(\delta_r(x))=r\:\dd(x)$, for all $r>0,\:x\in \G$,
	\item [(iii)]$\dd(x^{-1})=\dd(x)$, for all $x\in \G$,
	\item[(iv)]$\dd(x)=0$ if and only if $x=0$.
\end{enumerate}
A homogeneous norm usually only defines a quasi-metric on $\G$. Now for a choice of a collection of left invariant vector fields $\{X_1,\cdots,X_{N_1}\}$ that form a basis of the first stratification layer of $\mathfrak{g}$, one actually has a canonical choice of a left-invariant metric on $\G$, called the Carnot-Carath\'eodory metric, which we simply denote by $d$. The group $\G$ equipped with the metric $d$ is a complete, separable metric space. Corresponding to $d$, we have the induced homogeneous norm on $\G$,
\begin{equation*}
\dd(x):=d(0,x)\:,\: x\in \G\:.
\end{equation*}
For $x\in \G$ and $r>0$,  the Carnot-Carath\'eodory ball centred at $x$ with radius $r$ is defined as
\begin{equation*}
B(x,r):=\{y\in \G \mid \dd(y^{-1} \circ x)=d(x,y)<r\}\:.
\end{equation*}
From the above properties of the metric, it follows that $B(x,r)$ is the left translate by $x$ of the ball $B(0,r)$ which in turn, is the image under $\delta_r$ of the ball $B(0,1)$. For more details about Carnot-Carath\'eodory metrics, see \cite[pp. 232-236]{BLU}.

The above choice of $\{X_1,\cdots,X_{N_1}\}$ yields a canonical sub-elliptic operator,
\begin{equation*}
\mathcal{L}=\sum_{j=1}^{N_1}X_j^2,
\end{equation*}
which is called a sub-Laplacian for $\G$. The fundamental solution of the Heat operator corresponding to $\mathcal{L}$ is given by
	\begin{equation*}
	\Gamma(x,t;\xi):=\Gamma(\xi^{-1}\circ x,t),\:\:\:\: x\in \G,\:\xi\in \G,\:t\in(0,\infty),
	\end{equation*}where $\Gamma$ is a smooth, strictly positive function on $\G\times(0,\infty)$, as introduced in (\ref{rep_formula}). The kernel is normalized so that 
\begin{equation} \label{kernel_lebesgue}
\Gamma[m] \equiv 1 \:.
\end{equation}
For non-negative $f \in L^1(\G)$, the integral in (\ref{rep_formula}) for the measure $\mu=f\:dm$, is well-defined for all $(x,t) \in  \G \times (0,\infty)$ and will be denoted by $\Gamma[f]$, in stead of $\Gamma[f\:dm]$. Now by the pointwise estimates of $\Gamma$ given by \cite[Theorems 5.1, 5.2]{BLU1} and the fact that any two homogeneous norms on $\G$ are equivalent \cite[Proposition 5.1.4]{BLU}, there exists a constant $c_1 \ge 1$, depending only on the choice of $\{X_1,\cdots,X_{N_1}\}$ such that we have the following Gaussian estimates
\begin{equation} \label{Gaussian_estimates}
c^{-1}_1\:t^{-\frac{Q}{2}}\:\exp\left(-\frac{c_1 \dd(x)^2}{t}\right) \le \Gamma(x,t) \le c_1\:t^{-\frac{Q}{2}}\:\exp\left(-\frac{ \dd(x)^2}{c_1 t}\right)\:,
\end{equation} 
for every $x \in \G$ and $t \in (0,\infty)$.

The formula for integration in polar coordinates is given by \cite[Proposition 1.15]{Folland}: for all $f \in L^1(\G)$,
\begin{equation} \label{polar_coordinates}
\int_{\G}f(x)\:dm(x)=\int_{0}^\infty \int_S f\left(\delta_r(\omega)\right)\:r^{Q-1}\:d\sigma(\omega)dr\:,
\end{equation}
where $S=\left\{\omega \in \G \mid \dd(\omega)=1\right\}$, $\sigma$ is a unique positive Radon measure on $S$ and $dr$ is the restriction of the Lebesgue measure on $\R$ to $[0,\infty)$.

\subsection{Hausdorff dimensions:} In this section, we briefly recall the definitions of Hausdorff dimensions, Hausdorff measures. These can be found in \cite{F}. 

Let $(X,d)$ be a metric space. Then for a set $E\subset X$, $\ddd(E)$ will denote the diameter of $E$. For $\varepsilon > 0$, an $\varepsilon$-cover of $E \subset X$ is a countable (or finite) collection of sets $\{U_i\}$ with 
\begin{equation*}
0 < \ddd\left(U_i\right) \le \varepsilon \text{, for all }i \text{ such that } E \subset \displaystyle\bigcup_{i} U_i \:.
\end{equation*}
For $s \ge 0$, we recall that
\begin{equation*}
\mathcal{H}^s_\varepsilon(E) := \inf \left\{\displaystyle\sum_{i} {\left(\ddd\left(U_i\right)\right)}^s \mid \{U_i\} \text{ is an } \varepsilon\text{-cover of } E\right\}\:.
\end{equation*}
Then the $s$-dimensional Hausdorff outer measure of $E$ is defined by,
\begin{equation*}
\mathcal{H}^s(E) := \displaystyle\lim_{\varepsilon \to 0} \mathcal{H}^s_\varepsilon(E) \:.
\end{equation*}
The above value remains unaltered if one only considers covers consisting of balls. This outer measure restricts to a measure (also denoted by) $\mathcal{H}^s$ on a $\sigma$-algebra that contains all Borel sets. It is called the $s$-dimensional Hausdorff measure.

The Hausdorff dimension of $E$ is defined by,
\begin{equation*}
dim_{\mathcal{H}}E := \inf \left\{s \ge 0 \mid \mathcal{H}^s(E) =0\right\} \:.
\end{equation*}
If $s=dim_{\mathcal{H}}E$, then $\mathcal{H}^s(E)$ may be zero or infinite or $\mathcal{H}^s(E) \in (0,\infty)$\:. We will require the following result on existence of sets of desired Hausdorff dimensions, which is a special case of \cite[Corollary 7]{H}:
\begin{lemma} \label{existence}
If $Y$ is a Borel subset of a complete, separable metric space such that for some $s>0$, $\mathcal{H}^s(Y)=\infty$, then there exists a compact subset $Z \subset Y$ with $dim_{\mathcal{H}}Z=s$.
\end{lemma}

\section{Upper bound on the Hausdorff dimension}
In this section, we work under the assumption that $\mu$ is a Radon measure on $\G$ such that the integral, 
\begin{equation}\label{measure_assumption}
\Gamma[\mu](x,t)= \int_{\G} \Gamma(\xi^{-1} \circ x,t)\: d\mu(\xi) <+\infty\:, 
\end{equation}
for all $(x,t) \in  \G \times (0,\infty)$\:. We first prove the following auxiliary result:
\begin{lemma} \label{lemma1}
Let $\alpha \in [0,Q]$. Then for any $\varepsilon >0$ and any $x \in \G$,
\begin{equation*}
\lim_{t \to 0+} t^{\frac{\alpha}{2}} \int_{\G \setminus B(x,\varepsilon)} \Gamma(\xi^{-1} \circ x,t)\: d\mu(\xi) =0\:.
\end{equation*}
\end{lemma}
\begin{proof}
We fix $x \in \G$ and $\varepsilon >0$. Then by the upper bound in (\ref{Gaussian_estimates}),
\begin{equation*}
 t^{\frac{\alpha}{2}} \int_{\G \setminus B(x,\varepsilon)} \Gamma(\xi^{-1} \circ x,t)\: d\mu(\xi) \le c_1\:t^{-\frac{1}{2}\left(Q-\alpha\right)} \int_{d(x,\xi)\ge \varepsilon}\exp\left(-\frac{ d(x,\xi)^2}{c_1 t}\right)\:d\mu(\xi)\:.
\end{equation*}
Now for $t \in (0,1/2)$, the term on the right hand side of the above inequality can be further dominated as follows,
\begin{eqnarray*}
&&c_1\:t^{-\frac{1}{2}\left(Q-\alpha\right)} \int_{d(x,\xi) \ge \varepsilon}\exp\left(-\frac{ d(x,\xi)^2}{c_1 t}\right)\:d\mu(\xi) \\
& = & c_1\:t^{-\frac{1}{2}\left(Q-\alpha\right)} \int_{d(x,\xi) \ge \varepsilon}\exp\left(-\frac{ d(x,\xi)^2}{c_1}\left(\frac{1}{2}t^{-1}+\frac{1}{2}t^{-1}\right)\right)\:d\mu(\xi) \\
& \le & c_1\:t^{-\frac{1}{2}\left(Q-\alpha\right)} \int_{d(x,\xi) \ge \varepsilon}\exp\left(-\frac{ d(x,\xi)^2}{c_1}\left(\frac{1}{2}t^{-1}+1\right)\right)\:d\mu(\xi)\\
& \le & c_1\:t^{-\frac{1}{2}\left(Q-\alpha\right)}\:\exp\left(-\frac{ \varepsilon^2}{2c_1t} \right) \int_{\G}\exp\left(-\frac{ d(x,\xi)^2}{c_1}\right)\:d\mu(\xi)\:.
\end{eqnarray*}
Now by the lower bound in (\ref{Gaussian_estimates}),
\begin{eqnarray*}
&& c_1\:t^{-\frac{1}{2}\left(Q-\alpha\right)}\:\exp\left(-\frac{ \varepsilon^2}{2c_1t} \right) \int_{\G}\exp\left(-\frac{ d(x,\xi)^2}{c_1}\right)\:d\mu(\xi) \\
&\le & \left(c^{Q+2}_1\:\Gamma[\mu](x,c^2_1)\right)t^{-\frac{1}{2}\left(Q-\alpha\right)}\:\exp\left(-\frac{ \varepsilon^2}{2c_1t} \right)\:.
\end{eqnarray*}
Then in view of (\ref{measure_assumption}), the above inequality yields
\begin{equation*}
\lim_{t \to 0+} t^{\frac{\alpha}{2}} \int_{\G \setminus B(x,\varepsilon)} \Gamma(\xi^{-1} \circ x,t)\: d\mu(\xi) =0\:.
\end{equation*}
\end{proof}
This brings us to our next estimate, which relates small time asymptotics of positive solutions of the Heat equation, with the upper density of its boundary measure and is the heart of the matter in proving Theorem \ref{Hausdorff_bound}:
\begin{lemma} \label{lemma2}
Let $\alpha \in [0,Q)$. Then for any $x \in \G$,
\begin{equation} \label{limsup_estimate}
\limsup_{t \to 0+}\: t^{\frac{\alpha}{2}}\:\Gamma[\mu](x,t) \le c(\alpha)\:\limsup_{r \to 0}\: \frac{\mu\left(B(x,r)\right)}{r^{Q-\alpha}}\:,
\end{equation}
with
\begin{equation*}
c(\alpha)=c^{\left(1+\frac{Q-\alpha}{2}\right)}_1 \left(Q-\alpha\right)\int_0^\infty \exp(-r^2)\:r^{Q-\alpha-1}\:dr\:,
\end{equation*}
where $c_1$ is as in (\ref{Gaussian_estimates}).
\end{lemma}
\begin{proof}
We fix $x \in \G$. If the term on the right hand side of (\ref{limsup_estimate}) is $+\infty$, then the result follows automatically. So let us assume that there exists $\lambda>0$, such that
\begin{equation*}
c(\alpha)\:\limsup_{r \to 0}\: \frac{\mu\left(B(x,r)\right)}{r^{Q-\alpha}} < \lambda\:.
\end{equation*}
Then there exists $\varepsilon>0$, such that
\begin{equation} \label{limsup_eq1}
\mu\left(B(x,r)\right) < \frac{\lambda\: r^{Q-\alpha}}{c(\alpha)}\:,\text{ for } r \in (0,\varepsilon]\:.
\end{equation}
We now decompose,
\begin{eqnarray*}
t^{\frac{\alpha}{2}}\:\Gamma[\mu](x,t) &=& t^{\frac{\alpha}{2}} \int_{\G \setminus B(x,\varepsilon)} \Gamma(\xi^{-1} \circ x,t)\: d\mu(\xi) + t^{\frac{\alpha}{2}} \int_{B(x,\varepsilon)} \Gamma(\xi^{-1} \circ x,t)\: d\mu(\xi) \\
&=& I_1(t)\:+\:I_2(t)\:.
\end{eqnarray*}
By Lemma \ref{lemma1}, we have
\begin{equation} \label{limsup_eq2}
\lim_{t \to 0+} I_1(t)=0\:.
\end{equation}
We now focus on $I_2$. By the upper bound in (\ref{Gaussian_estimates}), we have
\begin{equation} \label{limsup_eq3}
I_2(t) \le  c_1\:t^{-\frac{1}{2}\left(Q-\alpha\right)} \int_{d(x,\xi)< \varepsilon}\exp\left(-\frac{ d(x,\xi)^2}{c_1 t}\right)\:d\mu(\xi)\:.
\end{equation}
Now setting 
\begin{equation*}
M(r):= \int_{d(x,\xi)< r} d\mu(\xi)=\mu\left(B(x,r)\right) \:\text{ and }\: F(r):= \exp\left(-\frac{r^2}{c_1 t}\right)\:, \text{ for } r > 0\:,
\end{equation*}
we note that 
\begin{equation*}
\int_{[0,r]}dM(s)=M(r)\:,
\end{equation*}
and by (\ref{limsup_eq1}), we have
\begin{equation} \label{limsup_eq4}
M(r) < \frac{\lambda\: r^{Q-\alpha}}{c(\alpha)}\:,\text{ for } r \in (0,\varepsilon]\:.
\end{equation}
We can now rewrite the term on the right hand side of the inequality (\ref{limsup_eq3}) as,
\begin{equation} \label{limsup_eq5}
c_1\:t^{-\frac{1}{2}\left(Q-\alpha\right)} \int_{d(x,\xi)< \varepsilon}\exp\left(-\frac{ d(x,\xi)^2}{c_1 t}\right)\:d\mu(\xi) = c_1\:t^{-\frac{1}{2}\left(Q-\alpha\right)} \int_{[0,\varepsilon]} F(r)\:dM(r)\:.
\end{equation}
We now perform an integration-by-parts, which is justified by the Fubini's theorem. Consider the integral
\begin{equation} \label{limsup_eq6}
\mathbb{I}=\int_{[0,\varepsilon]}\int_{s}^\varepsilon F'(r)\:dr\:dM(s)\:.
\end{equation}
 Integrating with respect to $s$ gives,
\begin{equation*}
\mathbb{I}=\int_0^\varepsilon F'(r)\:dr \int_{[0,r]}dM(s) = \int_0^\varepsilon F'(r)\:M(r)\:dr\:.
\end{equation*}
On the other hand, integrating (\ref{limsup_eq6}) with respect to $r$ gives,
\begin{equation*}
\mathbb{I}= \int_{[0,\varepsilon]} dM(s)\:\int_s^\varepsilon F'(r)\:dr = \int_{[0,\varepsilon]}\: \left(F(\varepsilon)-F(s)\right) dM(s)\:.
\end{equation*}
Equating the above two expressions, we obtain
\begin{eqnarray*}
 \int_{[0,\varepsilon]} F(r)\:dM(r)&=& \int_{[0,\varepsilon]}\: F(\varepsilon)\: dM(r) - \int_0^\varepsilon F'(r)\:M(r)\:dr \\
&=& F(\varepsilon)\:M(\varepsilon)  - \int_0^\varepsilon F'(r)\:M(r)\:dr\:.
\end{eqnarray*}
Now plugging the above in (\ref{limsup_eq5}), we get that
\begin{eqnarray*}
&& c_1\:t^{-\frac{1}{2}\left(Q-\alpha\right)} \int_{d(x,\xi)< \varepsilon}\exp\left(-\frac{ d(x,\xi)^2}{c_1 t}\right)\:d\mu(\xi)\\
& =& c_1\:t^{-\frac{1}{2}\left(Q-\alpha\right)} \left( F(\varepsilon)\:M(\varepsilon)  + \int_0^\varepsilon \left(-F'(r)\right)\:M(r)\:dr\right)\:.
\end{eqnarray*} 
Now as $F>0$ and $F'<0$, applying (\ref{limsup_eq4}) and repeating the integration-by-parts arguments, we get that 
\begin{eqnarray*}
&& c_1\:t^{-\frac{1}{2}\left(Q-\alpha\right)} \int_{d(x,\xi)< \varepsilon}\exp\left(-\frac{ d(x,\xi)^2}{c_1 t}\right)\:d\mu(\xi) \\
& < & \lambda \left(\frac{c_1\:t^{-\frac{1}{2}\left(Q-\alpha\right)} }{c(\alpha)}\right) \left[F(\varepsilon) \varepsilon^{Q-\alpha} \:-\: \int_0^\varepsilon  F'(r)\:r^{Q-\alpha}\:dr \right] \\
&=& \lambda \left(\frac{c_1\:t^{-\frac{1}{2}\left(Q-\alpha\right)} }{c(\alpha)}\right) \int_0^\varepsilon F(r)\:d(r^{Q-\alpha}) \:.
\end{eqnarray*}
Then plugging in the definition of $F$ and doing elementary change of variables yield
\begin{eqnarray*}
&& \lambda \left(\frac{c_1\:t^{-\frac{1}{2}\left(Q-\alpha\right)} }{c(\alpha)}\right) \int_0^\varepsilon F(r)\:d(r^{Q-\alpha})\\
& =& \lambda \left(\frac{c^{\left(1+\frac{Q-\alpha}{2}\right)}_1 \left(Q-\alpha\right) }{c(\alpha)}\right)\int_0^{\varepsilon/(c_1t)^{\frac{1}{2}}} \exp(-r^2)\:r^{Q-\alpha-1}\:dr \\
& \le & \lambda \left(\frac{c^{\left(1+\frac{Q-\alpha}{2}\right)}_1 \left(Q-\alpha\right) }{c(\alpha)}\right)\int_0^\infty \exp(-r^2)\:r^{Q-\alpha-1}\:dr\\
&=& \lambda\:.
\end{eqnarray*}
Hence we get that
\begin{equation} \label{limsup_eq7}
I_2(t) < \lambda\:.
\end{equation}
Then combining (\ref{limsup_eq2}) and (\ref{limsup_eq7}), it follows that
\begin{equation} \label{limsup_eq8}
\limsup_{t \to 0+}\: t^{\frac{\alpha}{2}}\:\Gamma[\mu](x,t)  \le \limsup_{t \to 0+}\: I_1(t) + \limsup_{t \to 0+}\: I_2(t) \le \lambda\:. 
\end{equation}
As $\lambda>0$, was arbitrarily chosen such that
\begin{equation*}
c(\alpha)\:\limsup_{r \to 0}\: \frac{\mu\left(B(x,r)\right)}{r^{Q-\alpha}} < \lambda\:,
\end{equation*} 
the estimate (\ref{limsup_eq8}) yields the desired result and completes the proof of Lemma \ref{lemma2}.
\end{proof}

We now present the proof of Theorem \ref{Hausdorff_bound}:

\begin{proof}[Proof of Theorem \ref{Hausdorff_bound}]
We first focus on estimating $\mathcal{H}^{Q-\alpha}\left(S_\alpha(u)\right)$. For $\alpha=0$, the result follows from the Fatou-Kato theorem of Bonfiglioli et. al. \cite[Theorem 1.1]{BU}. For $\alpha=Q$, the estimate (\ref{kernel_general_estimate}) implies that $S_Q(u)=\emptyset$ and hence the result follows. 

We now fix $\alpha \in (0,Q)$. Decompose
\begin{equation*}
\G = \bigcup_{j=1}^\infty E_j\:,
\end{equation*}
where
\begin{eqnarray*}
&&E_1 :=\left\{x \in \G \mid \dd(x) \le 1\right\}\:,\\
&&E_j := \left\{x \in \G \mid j-1 < \dd(x) \le j\right\}\:,\: j \ge 2\:.
\end{eqnarray*}
Now choose and fix $j \in \N$ and set
\begin{equation*}
S^j_\alpha(u):=S_\alpha(u)\cap E_j =\{x \in E_j \mid \displaystyle\limsup_{t \to 0+} t^{\alpha/2} u(x,t)=+\infty\}\:.
\end{equation*}
Furthermore, for each $k \in \N$, consider
\begin{equation*}
S^{j,k}_\alpha(u):=\{x \in E_j \mid \displaystyle\limsup_{t \to 0+} t^{\alpha/2} u(x,t)>k\}\:.
\end{equation*}
We now choose and fix $\varepsilon \in (0,1/2)$. Then by Lemma \ref{lemma2}, for each $x \in S^{j,k}_\alpha(u)$, there exists a Carnot-Carath\'eodory ball $B(x,r_x)$ with $r_x \le \varepsilon$ such that 
\begin{equation} \label{thm1eq1}
r^{Q-\alpha}_x<c(\alpha)\: \frac{\mu\left(B(x,r_x)\right)}{k}\:,
\end{equation}
where $\mu$ is the boundary measure of $u$, that is, $u=\Gamma[\mu]$\:. Now an application of Vitali 5-covering Lemma on the family of balls $\{B(x,r_x)\}_{x \in S^{j,k}_\alpha(u)}$, yields a countable subcollection of balls $\{B(x_l,r_l)\}_{l=1}^\infty$ satisfying (\ref{thm1eq1}) such that
\begin{itemize}
\item $r_l \le \varepsilon$, for all $l \in \N$,
\item $B(x_{l_p},r_{l_p}) \cap B(x_{l_q},r_{l_q}) = \emptyset$\:, for all $l_p \ne l_q$,
\item $S^{j,k}_\alpha(u) \subset \displaystyle\bigcup_{l=1}^\infty B(x_l,5r_l)$\:.
\end{itemize}
Next applying (\ref{thm1eq1}), we get the following estimate for the diameters of the enlarged balls,
\begin{equation*}
\sum_{l=1}^\infty \left(\ddd\left(B(x_l,5r_l)\right)\right)^{Q-\alpha} \lesssim \frac{1}{k} \sum_{l=1}^\infty \mu\left(B(x_l,r_l)\right)= \frac{1}{k}\: \mu\left(\bigcup_{l=1}^\infty B(x_l,r_l)\right)\:,
\end{equation*}
where the implicit constant depends only  on the parameter $\alpha$. Then as $\varepsilon \in (0,1/2)$ and $x_l \in S^{j,k}_\alpha(u)$ for all $l \in \N$, we note that
\begin{equation*}
\bigcup_{l=1}^\infty B(x_l,r_l) \subset
\begin{cases}
	 E_1 \cup E_2\:,&\text{ if } j=1\:,\\
	 E_{j-1} \cup E_j \cup E_{j+1}\:,&\text{ if } j \ge 2\:.
	\end{cases}
\end{equation*}
This yields that
\begin{equation*}
\sum_{l=1}^\infty \left(\ddd\left(B(x_l,5r_l)\right)\right)^{Q-\alpha} \lesssim \frac{1}{k}\:,
\end{equation*}
where the implicit constant depends only  on the parameter $\alpha$ and the choice of $j$. Hence letting $\varepsilon \to 0$, we get that
\begin{equation*}
\mathcal{H}^{Q-\alpha}\left(S^{j,k}_{\alpha}(u)\right) \lesssim \frac{1}{k}\:.
\end{equation*}
Now as 
\begin{equation*}
S^j_{\alpha}(u) \subset S^{j,k}_{\alpha}(u)\:, \text{ for all } k \ge 1\:, 
\end{equation*}
the above estimate implies that
\begin{equation*}
\mathcal{H}^{Q-\alpha}\left(S^j_{\alpha}(u)\right) =0\:.
\end{equation*}
Therefore, we get that
\begin{equation*}
\mathcal{H}^{Q-\alpha}\left(S_\alpha(u)\right)=\mathcal{H}^{Q-\alpha}\left(\bigcup_{j=1}^\infty S^j_{\alpha}(u)\right) \le \sum_{j=1}^\infty
\mathcal{H}^{Q-\alpha}\left(S^j_{\alpha}(u)\right) =0\:.
\end{equation*}

We are now only left to estimate $dim_{\mathcal{H}}\left(T_\alpha(u)\right)$. For $\alpha=0$, the result is trivial as  $dim_{\mathcal{H}}\left(\G\right)=Q$. Now for $0<\beta<\alpha \le Q$, we note that $T_\alpha(u) \subset S_\beta(u)$ and hence,
\begin{equation*}
\mathcal{H}^{Q-\beta}\left(T_\alpha(u)\right) \le \mathcal{H}^{Q-\beta}\left(S_\beta(u)\right)=0\:.
\end{equation*}
Therefore,
\begin{equation*}
dim_{\mathcal{H}}\left(T_\alpha(u)\right) \le \inf_{\beta \in (0,\alpha)}\:\left(Q-\beta\right)= Q-\alpha\:.
\end{equation*}
This completes the proof of Theorem \ref{Hausdorff_bound}\:.
\end{proof}

\section{The sharpness result}
In this section, we prove Theorem \ref{sharp_thm}. In order to do so, we first construct the building blocks of the divergence by means of a `uniform divergence' estimate for small time:
\begin{lemma} \label{heat_indicator_estimate}
There exists a constant $c_2 \in (0,1)$, depending only on the choice of $\{X_1,\cdots,X_{N_1}\}$ such that for all $x \in \G\:,\varepsilon>0$ and $t \in (0,c^2_2 \varepsilon^2]$,
\begin{equation*}
\Gamma\left[\chi_{B(x,\varepsilon)}\right](x,t) \ge \frac{1}{2} \:.
\end{equation*}
\end{lemma}
\begin{proof}
Choose and fix $t \in (0,c^2_2 \varepsilon^2]$. By (\ref{kernel_lebesgue}), we have
\begin{equation*}
\Gamma\left[\chi_{B(x,\varepsilon)}\right](x,t) = 1-  \int_{\G \setminus B(x, \varepsilon)} \Gamma(\xi^{-1} \circ x,t)\: dm(\xi) \:.
\end{equation*}
By the upper bound in (\ref{Gaussian_estimates}) and the symmetry and the left invariance of the Carnot-Carath\'eodory metric, we have
\begin{eqnarray*}
\int_{\G \setminus B(x, \varepsilon)} \Gamma(\xi^{-1} \circ x,t)\: dm(\xi) &\le & c_1\:t^{-\frac{Q}{2}}\int_{\G \setminus B(x, \varepsilon)}\exp\left(-\frac{ \dd(\xi^{-1} \circ x)^2}{c_1 t}\right)\:dm(\xi) \\
&= & c_1\:t^{-\frac{Q}{2}}\int_{\G \setminus B(0, \varepsilon)}\exp\left(-\frac{ \dd(\xi)^2}{c_1 t}\right)\:dm(\xi)\:. 
\end{eqnarray*}
Now by the integration in polar coordinates formula (\ref{polar_coordinates}), 
\begin{equation*}
c_1\:t^{-\frac{Q}{2}}\int_{\G \setminus B(0, \varepsilon)}\exp\left(-\frac{ \dd(\xi)^2}{c_1 t}\right)\:dm(\xi) = c_1\sigma(S)\:t^{-\frac{Q}{2}}\int_{\varepsilon}^\infty \exp\left(-\frac{r^2}{c_1 t}\right)\:r^{Q-1}\:dr\:.
\end{equation*}
After some elementary change of variables, binomial expansion and simplification of the integration in terms of Gamma functions, we get that 
\begin{eqnarray*}
c_1\sigma(S)\:t^{-\frac{Q}{2}}\int_{\varepsilon}^\infty \exp\left(-\frac{r^2}{c_1 t}\right)\:r^{Q-1}\:dr & \le & c_3\: \exp\left(-\frac{\varepsilon^2}{c_1t}\right) \displaystyle\sum_{k=0}^{Q-1} {\left(\frac{\varepsilon}{\sqrt{t}}\right)}^{2k-Q} \\
& \le & c_3\: \exp\left(-\frac{\varepsilon^2}{c_1t}\right) \left(\frac{\left(\frac{\varepsilon}{\sqrt{t}}\right)^Q}{\left(\frac{\varepsilon}{\sqrt{t}}\right)-1}\right)\:,
\end{eqnarray*}
for some $c_3>0$, depending only on the choice of $\{X_1,\cdots,X_{N_1}\}$. Now note that there exists $c_4>1$, sufficiently large, such that for 
\begin{equation*}
\frac{\varepsilon}{\sqrt{t}} \ge c_4 \:,
\end{equation*}
one has
\begin{equation*}
\left(\frac{\varepsilon}{\sqrt{t}} -1 \right) \ge  \frac{\varepsilon}{2\sqrt{t}} \:,\text{ and } \exp\left(-\frac{\varepsilon^2}{c_1t}\right) \: {\left(\frac{\varepsilon}{\sqrt{t}}\right)}^Q \le \frac{1}{2} \:.
\end{equation*}
Thus setting 
\begin{equation*}
c_2 = \min \left\{\frac{1}{c_4}, \frac{1}{2c_3}\right\} \:,
\end{equation*}
the result follows as
\begin{eqnarray*}
\int_{\G \setminus B(x, \varepsilon)} \Gamma(\xi^{-1} \circ x,t)\: dm(\xi) \le c_3 \left(\frac{\sqrt{t}}{\varepsilon}\right) \le \frac{1}{2} \:.
\end{eqnarray*}
\end{proof}

We now prove Theorem \ref{sharp_thm}:

\begin{proof}[Proof of Theorem \ref{sharp_thm}]
We first prove part $(i)$. Let $E \subset \G$ with $\mathcal{H}^{Q-\alpha}(E)=0$. Then for any $k \in \N$, there exists a covering of $E$ by Carnot-Carath\'eodory balls $\{B^k_j\}_{j=1}^\infty$ such that
their diameters satisfy
\begin{equation} \label{sharp_thm_eq1}
\displaystyle\sum_{j=1}^\infty {(\ddd\left(B^k_j\right))}^{Q-\alpha} < 2^{-k} \:.
\end{equation}
If $B$ is a Carnot-Carath\'eodory ball in $\G$ with center $x$ and radius $r$, then for notational convenience, $2B$ will denote the ball with center $x$ and radius $2r$. Now define,
\begin{equation*}
f:= \displaystyle\sum_{k=1}^\infty \displaystyle\sum_{j=1}^\infty k \:{\left(\ddd\left(B^k_j\right)\right)}^{-\alpha} \: \chi_{2B^k_j} \:.
\end{equation*}
Then by the homogenity of the Haar measure (\ref{homogenity}) and the estimate (\ref{sharp_thm_eq1}),  
\begin{eqnarray*}
\int_{\G} f\:dm & \le & \displaystyle\sum_{k=1}^\infty \displaystyle\sum_{j=1}^\infty k \:{\left(\ddd\left(B^k_j\right)\right)}^{-\alpha} \: m\left(2B^k_j\right) \\
& \lesssim & \displaystyle\sum_{k=1}^\infty \displaystyle\sum_{j=1}^\infty k \:{\left(\ddd\left(B^k_j\right)\right)}^{Q-\alpha} \\
& < & \displaystyle\sum_{k=1}^\infty k \:2^{-k} < \infty \:.
\end{eqnarray*}
Hence $f\:dm$ defines a finite Radon measure on $\G$ and thus $\Gamma[f]$ defines a positive solution of the Heat equation on $\G \times (0,\infty)$. 

Now let $x \in E$ and $k \in \N$. Then there exists $j_k \in \N$ such that $x \in B^k_{j_k}$. If $r_k$ is the radius of $B^k_{j_k}$ then 
\begin{equation*}
B(x,r_k) \subset 2B^k_{j_k}\:.
\end{equation*}
Then by the positivity of $\Gamma$ and Lemma \ref{heat_indicator_estimate}, one has
\begin{equation} \label{sharp_thm_eq2}
\Gamma\left[\chi_{2B^k_{j_k}}\right](x,t) \ge \Gamma\left[\chi_{B(x,r_k)}\right](x,t) \ge 1/2 \:, 
\end{equation}
whenever  $t \le c^2_2 r^2_k$\:. Thus defining $t_k := c^2_2 r^2_k$, we have by the definition of $f$ and (\ref{sharp_thm_eq2}),
\begin{eqnarray*}
\Gamma[f](x,t_k)  & \ge &  k {\left(\ddd\left(B^k_{j_k}\right)\right)}^{-\alpha} \: \Gamma\left[\chi_{2B^k_{j_k}}\right](x,t_k) \\
& \ge & k \left(2^{-(\alpha + 1)} c^{\alpha}_2\right) t^{-\alpha/2}_k \:.
\end{eqnarray*}
So there exists a positive constant $c_5$, depending only on the choice of $\{X_1,\cdots,X_{N_1}\}$ and the parameter $\alpha$, such that
\begin{equation*}
t^{\alpha/2}_k \: \Gamma[f](x,t_k) \ge c_5\: k \:.
\end{equation*}
As $t_k = c^2_2 r^2_k$, by (\ref{sharp_thm_eq1}), $t_k \to 0$ as $k \to +\infty$, and hence $x \in S_\alpha(\Gamma[f])$. Therefore, $E \subset S_\alpha(\Gamma[f])$. This gives part (i) of the result.

\medskip

For part $(ii)$, we start off by noting that for $\alpha=0$, any positive constant solution of the Heat equation yields the result. We next choose and fix $\alpha \in (0,Q)$. As any Carnot-Carath\'eodory ball $B$ is a Borel subset of the complete, separable metric space $(\G,d)$ with Hausdorff dimension $Q$, by Lemma \ref{existence}, for each integer $j > 1/(Q-\alpha)$, there exists a compact subset $K_j \subset B$ with 
\begin{equation} \label{partii_eq1}
dim_{\mathcal{H}}\left(K_j\right)=Q-\alpha-\frac{1}{j}\:.
\end{equation}
We now consider
\begin{equation*}
K:= \bigcup_j K_j\:.
\end{equation*}
Then for $s \ge Q - \alpha$, by (\ref{partii_eq1}),
\begin{equation} \label{partii_eq2}
\mathcal{H}^s(K) \le \sum_j \mathcal{H}^s(K_j) =0\:.
\end{equation}
On the other hand, for $s < Q - \alpha$, there exists $j_0 \in \N$, such that
\begin{equation*}
s < Q-\alpha-\frac{1}{j_0}\:,
\end{equation*}
and hence by (\ref{partii_eq1}),
\begin{equation} \label{partii_eq3}
\mathcal{H}^s(K) \ge \mathcal{H}^s(K_{j_0}) = \infty\:.
\end{equation}
Thus by (\ref{partii_eq2}) and (\ref{partii_eq3}), it follows that
\begin{equation} \label{partii_eq4}
dim_{\mathcal{H}}\left(K\right)= Q- \alpha\:,\text{ with } \mathcal{H}^{Q-\alpha}(K)=0\:.
\end{equation}
Then by part $(i)$ of Theorem \ref{sharp_thm}, there exists a positive solution of the Heat equation $u$ on $\G \times (0,\infty)$, such that $K \subset S_\alpha(u)\:.$ As by definition, $S_\alpha(u) \subset T_\alpha(u)$, we then have $K \subset  T_\alpha(u)$\:. Thus combining (\ref{partii_eq4}) and Theorem \ref{Hausdorff_bound}, we get that
\begin{equation*}
Q- \alpha = dim_{\mathcal{H}}\left(K\right) \le dim_{\mathcal{H}}\left(T_\alpha(u)\right) \le Q- \alpha\:.  
\end{equation*}
This completes the proof of Theorem \ref{sharp_thm}.
\end{proof}
\section*{Acknowledgements}
The author is supported by a research fellowship of Indian Statistical Institute.

\bibliographystyle{amsplain}

\end{document}